\documentclass[11pt]{amsart}
 \usepackage[T1]{fontenc}
\usepackage{graphicx}
\usepackage{url}

\input{xy}
\xyoption{all}
\xyoption{poly}
\usepackage[all]{xy}
\usepackage{float}
\usepackage{enumerate}
\usepackage[usenames]{color}
\usepackage[mathscr]{eucal}
 \usepackage{amsmath,amsthm,amsfonts,amssymb}
      \theoremstyle{plain}
      \newtheorem{theorem}{Theorem}[section]
      \newtheorem*{theorem*}{Theorem}
      
      \newtheorem{corollary}[theorem]{Corollary}
      \newtheorem{proposition}[theorem]{Proposition}

      \theoremstyle{definition}
	  \newtheorem{example}[theorem]{Example}
      \newtheorem{definition}[theorem]{Definition}
      
     \theoremstyle{remark}
      \newtheorem{remark}[theorem]{Remark}

 \newcommand\RR{{\mathbb{R}}}

 \def\I{{\mathcal I}}
 
 \def\K{{\mathcal K}}

 \def\N{{\mathscr N}}

 \def\R{{\mathcal R}}

 \def\U{{\mathcal U}}
 \def\V{{\mathcal V}}
 
 \def\X{{\mathcal X}}

\newcommand\gc{{\hspace{2.4 pt} \searrow\hspace{-8 pt}^{\gamma} \hspace{5 pt}}}

\newcommand\se{{\hspace{2 pt}\diagup\hspace{-4.8 pt} \searrow\hspace{5 pt}}}
\newcommand\co{{\hspace{2 pt}\searrow \hspace{3 pt}}}

\newcommand\sco{{\searrow \hspace{-6 pt}\searrow\hspace{3 pt}}}

      \makeatletter
      \def\@setcopyright{}
      \def\serieslogo@{}
      \makeatother
   
\begin{document}

\title {The cylinder of a relation and generalized versions of the Nerve Theorem}
   \author{Ximena Fern\'andez}
   \author{El\'ias Gabriel Minian}
   \address{Departamento  de Matem\'atica - IMAS\\
 FCEyN, Universidad de Buenos Aires. Buenos Aires, Argentina.}
\email{xfernand@dm.uba.ar ; gminian@dm.uba.ar}

   \begin{abstract}
   We introduce the notion of cylinder of a relation in the context of posets, extending the  construction of the mapping cylinder. We establish a local-to-global result for relations, generalizing Quillen's Theorem A for order preserving maps, and derive novel formulations of the classical Nerve Theorem for posets and simplicial complexes, suitable for covers with not necessarily contractible intersections.
     \end{abstract}



   \maketitle

   \section{Introduction}
The mapping cylinder $M(f)$ of a map $f:X\to Y$ between topological spaces is a tool that allows one to replace, up to homotopy, any continuous function $f$ by a nice inclusion, namely the cofibration $i:X\hookrightarrow M(f)$. The analogue of this construction in the context of posets, called the non-Hausdorff mapping cylinder, was introduced in \cite{BM08b} to study simple homotopy theory of polyhedra using the topology and combinatorics of finite topological spaces. The (non-Hausdorff) mapping cylinder $B(f)$ of an order preserving map $f:X\to Y$ between posets has similar properties as its classical analogue: $B(f)$ deformation retracts onto $Y$ (when they are viewed as finite topological spaces) and the inclusion $i:X\hookrightarrow B(f)$ satisfies nice properties. In \cite{Bar11b} Barmak used the poset version of the mapping cylinder to give an alternative and simple proof of Quillen's Theorem A for posets and to derive stronger formulations of the Nerve Theorem and other related results. 

 An order preserving map $f:X\to Y$ is a particular case of a relation $\R\subseteq X\times Y$ (a subset of the Cartesian product of the underlying sets) and it is natural to ask whether there exists a general construction of a cylinder of a relation that coincides with $B(f)$ in the case that the relation is the one induced by a poset map $f$. In this article we define the cylinder $B(\R)$ of any relation $\R\subseteq X\times Y$ and study its homotopical properties. Of course one cannot expect that $B(\R)$ preserves the homotopy type of $Y$ for a general relation $\R$ but we will show that this is the case when the relation satisfies good local properties (see Proposition \ref{colapsa Y}). Analogously, under symmetric local hypotheses one can show that $B(\R)$ preserves the homotopy type of $X$ (Proposition \ref{colapsa X}). As an immediate consequence we obtain a generalization of Quillen's Theorem A for relations (Theorem \ref{cilindro}). When the relation is the one induced by an order preserving map $f:X\to Y$, Theorem \ref{cilindro} is precisely  Quillen's Theorem A for posets.
  
The Nerve Theorem, due to Borsuk \cite{Bor48}, with alternative versions by 
Leray \cite{Ler45}, Weil \cite{Wei52} and  McCord \cite{McC67}, and more recently Bj\"orner \cite{Bjo03} and Barmak \cite{Bar11b}, is a powerful tool widely used in Topology, Combinatorics, and, in the last years, in Computational and Applied Topology (see \cite{Car09,dSG07,DMW17,NLC11}). The theorem essentially affirms that under certain hypotheses on a given space $X$, the nerve of a cover has the same homotopy type as $X$, provided that the cover is \textit{good}. Which means that
the intersection of any subfamily of the cover is either empty
or contractible. However the requirement of being a good cover is somewhat restrictive and not always convenient. 
In Section  \ref{section nerve} we prove generalizations of this theorem for not necessarily good covers. This is accomplished by using a formulation of the classical theorem in terms of posets (Theorem \ref{nerve poset}). The generalized versions of the Nerve Theorem are deduced by applying Theorem \ref{cilindro} to suitable relations.

Suppose that $K$ is a simplicial complex and $\U=\{L_i\}_{i\in I}$ is a cover of $K$ by subcomplexes, such that the intersections are not necessarily connected but the connected components of the non-empty intersections are contractible (this kind of situation appears for example in the study of point cloud data \cite{Car09,dSG07,SMC07}). In this case the usual nerve $\N(\U)$ is not in general homotopy equivalent to $K$. In Section \ref{section nerve} we introduce the completion of the nerve $\hat{\N}(\U)$, which is a regular CW-complex (not necessarily a simplicial complex) built from the connected components of the non-empty intersections. Corollary \ref{coro completion} shows that the completion has the same homotopy type as $K$.

\section{The cylinder of a relation}

Given a finite simplicial complex $K$ (or more generally, a finite regular CW-complex), its face poset will be denoted by $\X(K)$. This is the poset of simplices (or cells) of $K$ ordered by inclusion. Conversely, given a finite poset $X$, $\K(X)$ denotes its classifying space (also called the order complex). This is the simplicial complex whose simplices are the non-empty chains of $X$. A finite poset can be seen as a finite topological space whose open subsets are the downsets (see \cite{Bar11a,BM08b,McC66,Sto66}). Its topology is related to the topology of the order complex $\K(X)$ by means of the  \textit{McCord map} $\mu:\K(X)\to X$ which is a  weak (homotopy) equivalence  (i.e. a continuous map which induces isomorphisms in all homotopy groups) \cite{McC66}. There exists also a weak equivalence  $\tilde\nu:K\to \X(K)$ associated to the functor $\X$. Although in this article we will not adopt explicitly the finite space point of view,  we will use reduction methods of finite spaces to study homotopy properties of posets and complexes. We recall below the reduction methods that we will use and fix some notations. For more details the reader may consult \cite{Bar11a}.

Given a finite poset $X$ and $x \in X$,  we denote by
$U_x\subseteq X$ the subposet of elements which are smaller than or equal to $x$ (this is the minimal open subset containing $x$, if we view $X$ as a finite space). Analogously, we denote by
$F_x$  the subposet of elements of $X$ which are greater than or equal to $x$ (this corresponds to the closure of $\{x\}$ in the finite space $X$). When a point $x$ belongs to different posets $X,Y$, sometimes we write $U_x^X,U_x^Y$ (resp. $F_x^X,F_x^Y$) to distinguish whether the subposets are considered in $X$ or in $Y$.

Note that a poset $X$ is homotopically trivial (which means that all its homotopy groups are trivial) 
if and only if $\K(X)$ is a contractible polyhedron. A finite poset is called a \it finite model \rm of a CW-complex $K$ 
 if its order complex $\K(X)$ is homotopy equivalent to $K$. The first reduction method for finite spaces was introduced by Stong \cite{Sto66}. 
Given a finite poset $X$, an element $x\in X$ is called an
\it up beat point \rm
if $\hat F_x = F_x\smallsetminus\{x\}$ has a minimum, i.e. there is a unique element
$y\in X$ which covers $x$. Similarly $x$ is called a
\it down beat point \rm if $\hat U_x = U_x \smallsetminus\{x\}$ has maximum 
(there is a unique $y$ covered by $x$). If $x$ is a beat point (up or down), $X\smallsetminus\{x\}\subseteq X$ is a strong deformation retract. This reduction (i.e. the removing of the point $x$) is denoted by $X\sco X\smallsetminus\{x\}$.
Stong proved that a finite poset $X$ is 
contractible if and only if one can remove beat points from $X$, one by one,
to obtain a poset with only one element  (this is also called \textit{dismantlable}, see for example \cite{DR81}). If $X$ is contractible then it
is homotopically trivial, but the converse does not hold 
(see \cite[Section 1.3]{Bar11a} for more details). Note that the classical Whitehead's Theorem is not valid in the context of finite spaces (there are finite spaces which are weak equivalent but not homotopy equivalent).

The notion of \it collapse \rm in the context of posets was 
introduced by Barmak and Minian in \cite{BM08b} and it corresponds
to Whitehead's notion of simplicial collapse \cite{Coh70, Whi39}.  Recall that a simplex $\sigma$ of a simplicial complex $K$ is a \textit{free face} of $K$ if there is a unique simplex $\tau \in K$ containing $\sigma$ properly. In that case, there is a collapse from $K$ to the subcomplex $L=K \setminus \{\sigma, \tau\}$ (this is denoted by $K\co L$), and two simplicial complexes $K,T$ are simple homotopy equivalent if one can obtain one from the other by performing collapses and expansions (we denote $K \se T$ following the standard notation). If $K\se T$, they are in particular homotopy equivalent.
A point $x$ of a finite poset $X$ is a 
\it down weak point \rm if $\hat U_x$ is contractible, and it is 
an \it up weak point \rm if  $\hat F_x$ is contractible. An \textit{elementary collapse} is the deletion of a weak point. The inverse operation is called an \textit{elementary
 expansion}.  We say that $X$ \textit{collapses} to $Y$ (or
$Y$ \textit{expands} to $X$), and write $X \co Y$,
 if there is a sequence of elementary collapses which
starts in $X$ and ends in $Y$. A poset $X$ is said to be
\it collapsible \rm if it collapses to
 a point. Note that any beat point is in particular a weak point (since posets with maximum or minimum are contractible). Therefore 
 if $X$ is a contractible poset then it is collapsible. 
 In \cite[Sections 4.2 and 4.3]{Bar11a} there are 
 various examples of collapsible posets which are not contractible. 
 Finally, we say that $X$ is \textit{simply equivalent} to $Y$ if there
 exists a sequence of collapses
and expansions that starts in $X$ and ends in $Y$. This is denoted by 
 $X\se Y$. In \cite{BM08b} it is proved that $X\se Y$ 
 if and only if the associated order complexes are simple homotopy equivalent. 
 Moreover, if $X\co Y$ then $\K(X)$ simplicially collapses to $\K(Y)$.  Conversely, $K\se L$ if and only if $\X(K) \se \X(L)$. 
 It can be also shown that if 
 $\hat U_x$ or  $\hat F_x$ is a collapsible poset then 
 $\K(X)\co \K(X\smallsetminus\{x\})$ (see \cite[Remark 4.3.1]{BM08b}).

Finally, an element $x\in X$ is called a \textit{$\gamma$-point} if  $\hat U_x$ or $\hat F_x$ is homotopically trivial. In \cite{Bar11a} it is proved that, in this case, the inclusion induces a simple homotopy equivalence $\K(X)\se \K(X\smallsetminus\{x\})$. This reduction in denoted by 
 $X \gc X\smallsetminus \{x\}$.

The relationship between the homotopy theory and simple homotopy theory of posets and polyhedra is summarized in the following theorem. The proof can be found in \cite{Bar11a,BM08b}. We will use this relationship throughout the paper to derive results on complexes using posets and vice-versa.

\begin{theorem}\label{relationship}
Let $X$ and $Y$ be finite posets and $K$ and $L$ be finite simplicial complexes (or more generally, finite regular CW-complexes).
\begin{enumerate}
\item $X$ and $Y$ are weak equivalent if and only if $\K(X)$ and $\K(Y)$ are homotopy equivalent.
\item $X\se Y$ if and only if $\K(X)\se\K(Y)$. Moreover, $X\co Y$ implies $\K(X)\co\K(Y)$.
\item $X\se X'=\X(\K(X)$ (the barycentric subdivision of the poset $X$).
\item $K$ and $L$ are homotopy equivalent if and only if $\X(K)$ and $\X(L)$ are weak equivalent.
\item $K\se L$ if and only if $\X(K)\se\X(L)$. Moreover, $K\co L$ implies $\X(K)\co\X(L)$.
\item $K\se K'=\K(\X(K))$ (the barycentric subdivision of $K$).
\end{enumerate}
\end{theorem}

The \textit{opposite} of a poset $X$ will be denoted by $X^{op}$. This is the poset with the same underlying set as $X$ but with the opposite partial order. Note that $X\se X^{op}$ since their order complexes are isomorphic. A linear extension of a poset $Y$ is a total ordering $y_1,y_2,\ldots,y_r$ of the elements of $Y$ such that, if $y_i\leq y_j$, then  $i\leq j$.

The  \textit{non-Hausdorff mapping cylinder} of an order preserving map $f:X\rightarrow Y$
between finite posets is the poset $B(f)$ 
with underlying set the disjoint union $X\sqcup Y$, and order relation
defined by keeping the original ordering in $X$ and $Y$, and setting for each $x\in X$ and  $y\in Y$,  $x\leq y$ in $B(f)$ if $f(x)\leq y$ in $Y$.
The mapping cylinder $B(f)$ has similar properties as the classical mapping cylinder.  It is not difficult to see that $B(f)\sco Y$ and therefore, to study homotopy properties of posets, one can replace any order preserving map $f:X\to Y$ by the inclusion $i:X\hookrightarrow B(f)$. Under nice properties on $f$, $B(f)$ has 
the same topological properties as $X$.

Given two finite posets $X$ and $Y$, a relation $\R\subseteq X\times Y$ will mean a relation between
their underlying sets (i.e. a subset  $\R\subseteq X\times Y$ of the Cartesian product). We write $x\R y$ if $(x,y)\in \R$. 

 \begin{definition} \label{BR} Given a relation $\R\subseteq X\times Y$ between two finite posets, we define  the \textit{cylinder of the relation} $B(\R)$
 as the following finite poset. The underlying set is the disjoint
 union $X \sqcup Y$. We keep the given ordering 
 in both $X$ and $Y$, and for every $x\in X$ and $y\in Y$
 we set $x\leq y$ in $B(\R)$ 
if there are  points  $x'\in X$ and $y'\in Y$ such that 
$x\leq x'$ in $X$, $y'\leq y$ in $Y$ and $x'\R y'$ (i.e.
we take the order relation generated by $x\leq y$ if $x\R y$). 
\end{definition}

  For $A\subseteq X$ and $B\subseteq Y$,
  we set \begin{gather*}\R(A)=\{y\in Y: x\R y \text{ for some }x\in A\},\\
 \R^{-1}(B)=\{x\in X: x\R y \text{ for some }y\in B\}.\end{gather*}

\begin{remark}
An order preserving map $f:X\to Y$  induces a relation  
$\R\subseteq X\times Y$ defined by $x\R f(x)$ for every $x\in X$. In this case, the relation cylinder 
$B(\R)$ coincides with the non-Hausdorff mapping cylinder $B(f)$. Note also that, in this case, $\R(A)=f(A)$ and $\R^{-1}(B)=f^{-1}(B)$.
\end{remark}

Let $A$ be a subposet of $X$, we denote 
\begin{gather*}\overline{A} = \{x \in X :
x \geq a \text{ for some }a \in A \}=\bigcup_{a\in A } F_a \text{ (the closure of
$A$), and }\\
\underline{A} = \{x \in X :x \leq a \text{ for some }a \in A\} =
\bigcup_{a\in A } U_a\text{ (the
open hull of $A$).}\end{gather*}

The next results relate the
homotopy properties of $B(\R)$ with those of $X$ and $Y$. As a corollary we obtain a generalization of Quillen's Theorem A. The proofs are based on the standard reduction methods of finite spaces mentioned above. 

 \begin{proposition}\label{colapsa X}Let $\R\subseteq X\times Y$ be
 a relation between finite
 posets.  If the open hull $\underline{\R^{-1}(U_y)}$  is 
 homotopically trivial
  for every $y\in Y$, then $B(\R)$ and $X$ are weak equivalent. Moreover
the inclusion  $X\hookrightarrow B(\R)$ induces a simple homotopy equivalence. In particular, the order complexes $\K(B(\R))$ and $\K(X)$ are (simple) homotopy equivalent.
 Further, if for every
  $y\in Y$,  $\underline{\R^{-1}(U_y)}$ 
 is collapsible, then 
 $\K(B(\R)) \searrow \K(X)$.
 \end{proposition}
 
 \begin{proof} 
 Let $y_1, y_2,  \cdots, y_n$ be a linear extension 
of the poset $Y$.  We will show 
that $$B(\R)\gc B(\R)\smallsetminus\{y_1\}\gc\cdots \gc B(\R)\smallsetminus
 \{y_1, y_2, \cdots, y_n\}=X.$$
For $2\leq i \leq n$, \[\hat U_{y_{i}}^{B(\R)\smallsetminus \{y_1, y_2, 
\cdots, y_{i-1}\}   
}=\hat U_{y_i}^{B(\R)}\smallsetminus Y=\underline{\R^{-1}(U_{y_i}^Y)},\]
 which is homotopically trivial by hypothesis. 
It follows that $y_i$ is a $\gamma$-point of the poset
 $B(\R)\smallsetminus \{y_1, y_2, 
\cdots, y_{i-1}\}$.
 Therefore, the inclusion induces a simple homotopy equivalence  $B(\R)\se X$.
 
 If for every
  $y\in Y$,  $\underline{\R^{-1}(U_y)}$ 
 is collapsible, then the previous collapses induce collapses
 between 
 the associated
 simplicial complexes by \cite[Remark 4.3.1]{Bar11a}.
 \end{proof}
 
 Analogously one can prove the following.
 
 \begin{proposition}\label{colapsa Y}Let $\R\subseteq X\times Y$ be
 a relation.  If the closure $\overline{\R(F_x)}$ is
  homotopically trivial
  for every $x\in X$, then 
 $B(\R) \se Y$. In particular, $\K(B(\R))$ and $\K(Y)$ are
 simple homotopy equivalent. Moreover, if for
  every $x\in X$, 
 $\overline{\R(F_x)}$ is collapsible,
  then $\K(B(\R)) \searrow \K(Y)$.
 \end{proposition}

As a corollary of the previous results we obtain a generalization of Quillen's Theorem A for relations $\R\subseteq X\times Y$.

\begin{theorem}\label{cilindro}Let $\R\subseteq X\times Y$ be
 a relation between finite posets.  If $\underline{\R^{-1}(U_y)}$ and $\overline{\R(F_x)}$ are
  homotopically trivial
  for every $x\in X$ and $y\in Y$, then 
 $X \se Y$. In particular,
 $\K(X)$ and $\K(Y)$ are (simple) homotopy equivalent.
\end{theorem}

If $\R$ is the relation 
associated to an order preserving map $f:X\to Y$, then $\underline{\R^{-1}(U_y)}=f^{-1}(U_y)$ and  $\overline{\R(F_x)}=F^Y_{f(x)}$. Note that, for any $x$, $F^Y_{f(x)}$ is contractible since it has a minimum. Therefore, the hypotheses of Proposition \ref{colapsa Y} are automatically fulfilled and one can deduce Quillen's Theorem A \cite{Qui78} (cf. \cite[Thm.1.2]{Bar11b}). 

\begin{theorem}[Quillen's Theorem A for posets]
Let $f:X\to Y$ be an order preserving map between finite posets. If $f^{-1}(U_y)$
is homotopically trivial for every $y$ in $Y$, then $\K(f):\K(X)\to\K(Y)$ is a simple homotopy equivalence. 
\end{theorem}

By relaxing the hypotheses on $\R$, we obtain weaker versions of Theorem \ref{cilindro}.

\begin{proposition}\label{relation homology}Let $\R\subseteq X\times Y$ be
 a relation between 
 posets. 
  If $\underline{\R^{-1}(U_y)}$, $\overline{\R(F_x)}$ are
 $n$-connected (resp. have trivial reduced $k$-homology groups for all $k\leq n$)
  for every $x\in X$ and $y\in Y$, then 
 $\pi_i(X)=\pi_i(Y)$ (resp. $H_i(X)=H_i(Y)$) for all 
 $0\leq i\leq n$.
\end{proposition}

\begin{proof}
We follow the proofs of Propositions \ref{colapsa X} and \ref{colapsa Y}.
If $\underline{\R^{-1}(U_{y_i})}$ is $n$- connected, then 
the inclusion $i:B(\R)\smallsetminus \{y_1, \cdots, y_{i}\} 
\hookrightarrow B(\R)\smallsetminus \{y_1, \cdots, y_{i-1}\}$ is 
an $(n+1)$- equivalence by  \cite[Lemma 6.2]{Bar11b}.
Then $i:X\hookrightarrow B(\R)$ is an $(n+1)$-
equivalence. Similarly, $j:Y\hookrightarrow B(\R)$ is an $(n+1)$- equivalence. 

The homology case is similar, using \cite[Lemma 6.3]{Bar11b}.
\end{proof}

\section{Generalizations of the Nerve Theorem}\label{section nerve}

Recall that the \textit{nerve} of a family $\U = \{U_i\}_{i\in I}$
of subsets of a set $X$ is the simplicial complex $\N(\U)$ whose simplices are the finite subsets $J\subseteq I$ such that 
$\displaystyle \bigcap_{i\in J} U_i \neq \varnothing$.

There exist different versions of the Nerve Theorem, 
involving open sets of topological spaces, subcomplexes of CW-complexes, etc. 
We mention here one version suitable for our purposes (see \cite{Bar11b}).

\begin{theorem}[Nerve Theorem for complexes]\label{nerve cs} Let $K$ be a finite
simplicial complex (or more generally, a regular CW-complex) and let $\U = 
\{L_i\}_{i\in I}$ be a family of subcomplexes which cover $K$ (i.e. 
$ \bigcup_{i\in J} L_i = K$).
If every intersection of elements of
$\U$ is empty or contractible, then $K$ and $\N(\U)$ have the same homotopy 
type. Moreover $K\se \N(\U)$.
\end{theorem}

The previous theorem can be restated in terms of finite posets by means of the functors $\X$ and $\K$. The equivalence between both formulations follows from Theorem \ref{relationship}. Note that covers  by subcomplexes correspond to open covers of posets. Recall that a subposet of a poset $X$ is open if it is a downset. Given an open cover $\U=\{U_i\}_{i\in I}$  of a finite poset $X$, we write $\X(\U)=\X(\N(\U))$. Note that  the elements of the poset $\X(\U)$ are the subsets $J\subseteq I$ and the partial order is given by the inclusion.

\begin{theorem}[Nerve Theorem for posets] \label{nerve poset} If $X$ is a finite poset and  $\U =
\{U_i\}_{i\in I}$ is an open cover of $X$ such that
every intersection of elements of
$\U$ is empty or homotopically trivial, then $X$ and $\X(\U)$ are weak equivalent. Moreover 
 $X \se\X(\U)$. 
 \end{theorem}

Covers satisfying the hypotheses of Theorem \ref{nerve cs} or Theorem \ref{nerve poset} are called \textit{good}. We use the formulation of the Nerve Theorem in terms of posets and Theorem \ref{cilindro} to derive generalizations of the result for covers which are not necessarily good. We introduce first some notations.

Let  $\U=\{U_i\}_{i\in I}$ be an open cover of $X$. Given $J\subseteq I$ we denote by $W_J$ the intersection $\bigcap_{i\in J} U_i$. Note that $W_J$ is open in $X$ for every $J\subseteq I$. We denote by $\X_0(\U)$ the subposet of $\X(\U)$ consisting of all $J\subseteq I$ such that $W_J$ is homotopically trivial. Note that $\X_0(\U)=\X(\U)$ if and only if $\U$ is a good cover.
 If $x\in X$ we denote by $I_x$ the subposet of $\X_0(\U)$ consisting of all $J\in\X_0(\U)$ such that $x\in W_J$. Note that $I_x$ is open in $\X_0(\U)$.

\begin{theorem}\label{nerve general etf}
 Let $X$ be a finite poset and let $\U=\{U_i\}_{i\in I}$ be an open 
 cover
 of $X$. If for every $x\in X$, the subposet $\I_x$ of
$\X_0(\U)$ is homotopically trivial, then
 $X \se \X_0(\U)$. 
\end{theorem}

\begin{proof}
Consider the opposite poset $\X_0(\U)^{op}$ and the relation $\R\subseteq X\times \X_0(\U)^{op}$ defined by:  $$x~\R~ J\text{ if }x \in W_J .$$

Let $J\in\X_0(\U)^{op}$. If $J'\in U_J$, then $J\subseteq J'$ and hence $W_{J'}\subseteq W_J$. It follows that $\underline{\R^{-1}\left(U_J\right) }=
\underline{W_J}=W_J$, which is homotopically trivial by definition of $\X_0(\U)$.

On the other hand, for every $x\in X$, $\overline{\R(F_x)}=\overline{I_x^{op}}=I_x^{op}$, which is homotopically trivial by hypothesis. By Theorem \ref{cilindro}, $X\se  \X_0(\U)^{op}$. 
Since $\X_0(\U)^{op}\se \X_0(\U)$, 
we deduce that $X\se  \X_0(\U)$. 
\end{proof}

\begin{remark} 
Note that the Nerve Theorem \ref{nerve poset} is a particular case of Theorem \ref{nerve general etf}. 
If $\U=\{U_i\}_{i\in I}$ is a good open cover of a finite poset $X$ (i.e $\X(\U)=\X_0(\U)$), then
 for every $x\in X$ the subposet $\I_x$ has a maximum element, namely $J=\{j\in I,\ x\in U_j\}$. In particular every $I_x$ is homotopically trivial and Theorem \ref{nerve general etf} applies. 
\end{remark}

By relaxing the hypothesis on the subposets $I_x$, one can obtain variations of Theorem \ref{nerve general etf} using Proposition \ref{relation homology}.

Suppose now that we are given a cover such that the intersections consist of disjoint unions of contractible subcomplexes (or homotopically trivial open subposets). For example, let $K$ be the boundary of the $2$-simplex with vertices $u,v,w$. Let $L$ be the subcomplex of $K$ consisting of the edges $\{u,w\}$ and $\{v,w\}$ and let $T$ be the edge $\{u,v\}$. The cover $\U=\{L,T\}$ is not good and clearly $\N(\U)$ is not homotopy equivalent to $K$, but the intersection $L\cap T$ is the union of two contractible subcomplexes (in this case, two points). We will show that in this situation we can replace the nerve $\N(\U)$ by a regular CW-complex $\hat{\N}(\U)$, which is homotopy equivalent to $K$ (see Example \ref{completion} below). The regular CW-complex $\hat{\N}(\U)$ will be the  \textit{completion} of the nerve. We study first the problem in the context of posets and then we derive the result for complexes.

\begin{definition}\label{quasigood}
A cover $\U$ of open subposets (or subcomplexes) is 
called 
\textit{quasi-good} if every non-empty intersection of a subfamily of $\U$
has homotopically trivial connected components. 
\end{definition}

\begin{definition}
Let $X$ be a finite poset and let $\U=\{U_i\}_{i\in I}$ be an open cover of $X$. Recall that for each $J\subseteq I$, $W_J$ denotes the intersection  $\bigcap_{i\in J} U_i$. The \textit{completion} of the nerve $\hat{\X}(\U)$ is the poset whose elements are the pairs

$$\hat{\X}(U)=\{(J,C) : \ J\subseteq I \text{ with } W_j\neq \emptyset,\ C \text{ a connected component of }W_J\}.$$
The order is given by $(J,C)\leq (J',C')$ if $J\subseteq J'$ and $C'\subseteq C$.
\end{definition}

\begin{remark}
If all non-empty intersections of the cover $\U$ are connected, then $\hat{\X}(\U)=\X(\U)$. In particular, if $\U$ is a good cover, $\X(U)$ coincides with its completion.
\end{remark}

\begin{theorem}\label{quasi good cover etf}
 Let $X$ be a finite poset and let $\U$ be a quasi-good cover 
 of $X$.
Then, 
$X \se  \hat{\X}(\U)$. 
\end{theorem}

\begin{proof}
Consider the opposite poset $\hat{\X}(\U)^{op}$ and the relation $\R\subseteq X\times \hat{\X}(\U)^{op}$ defined by:  $$x~\R~ (J,C)\text{ if }x \in C .$$

Let $(J,C)\in \hat{\X}(\U)^{op}$. If $(J',C')\in U_{(J,C)}$ then $J\subseteq J'$ and $C'\subseteq C$. It follows that $\underline{\R^{-1}\left(U_{(J,C)}\right) }=
\underline{C}=C$, which is homotopically trivial by hypothesis. Note that $C$ is open since it is a connected component of an open subposet.

On the other hand, for every $x\in X$, take $(J,C)\in \hat{\X}(\U)$ where $J\subseteq I$ is the maximum subset of $I$ such that $x\in W_J$ and $C$ is the unique component of $W_J$ containing $x$.  We will show that  $\overline{\R(F_x)}=F_{(J,C)}$ the closure of the element $(J,C)$ in $\hat{\X}(\U)^{op}$, which is contractible (since it has a minimum).
To see this, note first that $(J,C)\in \overline{\R(F_x)}$ and therefore the closure $F_{(J,C)}$ is contained in  $\overline{\R(F_x)}$. The other inclusion follows by the choice of $(J,C)$. If $y\in F_x$ then $x\leq y$. Therefore for any $(J',C')$ such that $y\in C'$ we have $(J,C)\leq (J',C')$ in $\hat{\X}(\U)^{op}$ since $C'$ is open. It follows that $\R(F_x)\subseteq F_{(J,C)}$ and therefore $\overline{\R(F_x)}\subseteq F_{(J,C)}$.

By Theorem \ref{cilindro}, $X\se  \hat{\X}(\U)^{op}$, and thus $X\se  \hat{\X}(\U)$. 
\end{proof}

From Theorem \ref{quasi good cover etf} we can derive a similar result for polyhedra. Note first that for any open cover $\U$ of a finite poset $X$, the completion $\hat{\X}(\U)$ is the face poset of a regular CW-complex (not necessarily a simplicial complex) whose  cells are simplices. To see this, note that for every $(J,C)\in \hat{\X}(\U)$ the subposet $U_{(J,C)}$ is the face poset of a simplex (of dimension $\# J-1$) since for any $J'\subseteq J$ there exists a unique component $C'\subseteq W_{J'}$ such that $C\subseteq C'$. The $n$-cells of this regular CW-complex correspond to the connected components of the $W_J$ with  $\# J=n+1$.


\begin{definition}
Let $K$ be a simplicial complex (or more generally, a regular CW-complex) and let $\U=\{L_i\}_{i\in I}$ be a cover of $K$ by subcomplexes. Consider the open cover  $\U_{\X}=\{\X(L_i)\}_{i\in I}$ of the face poset $\X(K)$. We define the \textit{completion} $\hat{\N}(\U)$ as the regular CW-complex satisfying $\X(\hat{\N}(\U))=\hat{\X}(\U_{\X})$. Concretely,  $\hat{\N}(\U)$ is the regular CW-complex whose cells (simplices) correspond to the connected components of the intersections of the elements of $\U$.
\end{definition}

\begin{corollary}\label{coro completion}
Let $K$ be a finite simplicial complex (or a regular CW-complex) and let 
$\U = \{L_i\}_{i\in I}$ be a quasi-good cover of 
$K$. Then $K\se \hat{\N}(\U)$.
\end{corollary}

\begin{proof}
Consider the open cover $\U_{\X}=\{\X(L_i)\}_{i\in I}$ of the face poset $\X(K)$. By Theorem \ref{quasi good cover etf}, $\X(K)\se \hat{\X}(\U_{\X})$. Then $$K'=\K(\X(K))\se \K(\hat{\X}(\U_{\X}))=\K(\X(\hat{\N}(\U)))=(\hat{\N}(\U))'.$$ Therefore $K\se \hat{\N}(\U)$.
\end{proof}

The following simple example illustrates the difference between the nerve and its completion.

\begin{example} \label{completion}
If $K$ is the boundary of the $2$-simplex and the cover $\U=\{L,T\}$ is defined as in the paragraph preceding Definition \ref{quasigood}, the completion $\hat{\N}(\U)$ is the regular CW-complex homeomorphic to $S^1$ with two $0$-cells and two $1$-cells

\begin{figure}[H]
\includegraphics[scale=0.27]{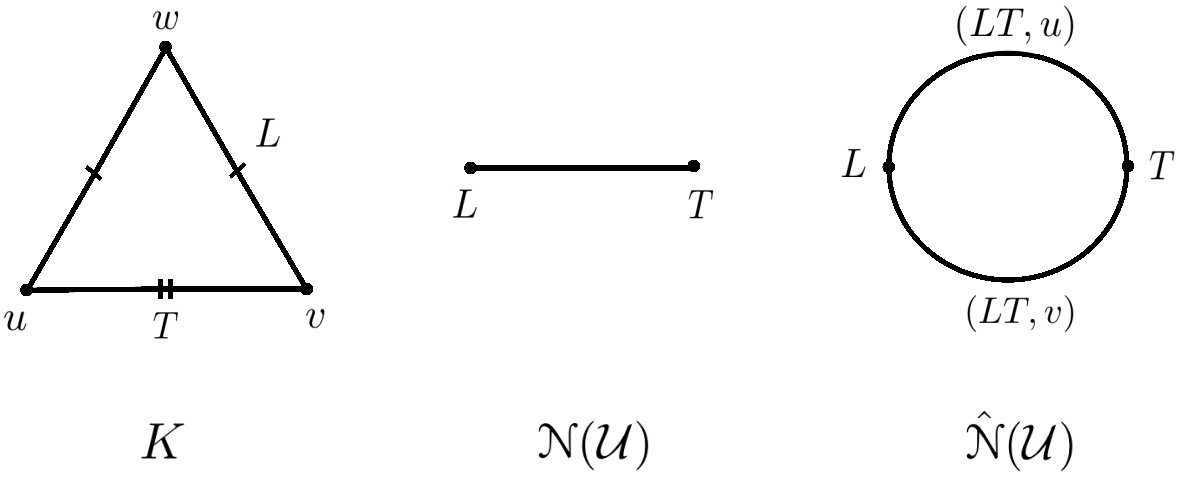}
\caption{The completion of the nerve of a quasi-good cover.}
\end{figure}
\end{example}

\begin{remark}
The completion of the nerve of a cover is related to the Mapper algorithm introduced in \cite{SMC07} for the shape recognition of data (see also 
\cite[Section 3.2]{Car09}). The Mapper algorithm uses the following construction. 
Suppose we are given a topological space $X$ and a  continuous map $f:X\to Z$
where $Z$ is a metric space (usually $\RR$ or $\RR^d$) and $\U$ is an open 
cover of $Z$ (when $Z=\RR$, $\U$ is usually taken as a family of overlapping 
open intervals). By pulling back the cover $\U$, we obtain an open cover 
$\V=f^*\U$ of $X$.  Note that the elements of $\V$ are not necessarily  path connected. 
Consider the open cover $\overline\V$ of $X$ consisting of the path connected components of the elements of $\V$ and take the nerve $\N(\overline\V)$. As it is explained in \cite[Section 3.2]{Car09}, the nerve $\N(\overline\V)$ (denoted by  $\check{C}^{\pi_0}(\V)$ in \cite{Car09}) gives a better approximation of $X$ than the usual nerve $\N(\V)$.  This construction provides a way to obtain   a cover of $X$ by means of $f:X\to Z$, even when we have incomplete information about 
 the topology of $X$. It is commonly used in data analysis, when we are only
given a finite sample of $X$  and we want to
deduce the homotopy type of $X$ from this sample.
Our construction of the completion of the nerve goes in this direction. Moreover, Corollary \ref{coro completion} suggests that a better approximation of $X$ is obtained if one takes $\hat \N(\V)$ (i.e. considering the components of all non-empty intersections). When the nerve of the cover is $1$-dimensional (for example when $Z=\RR$ and $\U$ is a family of intervals with no threefold overlaps), 
the completion is attained by considering the components of the elements of the cover and the components of the intersection of each pair of elements. In Example \ref{completion} the nerve of the cover $\overline\U$ coincides with the usual nerve $\N(\U)$ (since $L$ and $T$ are connected), but the completion $\hat{\N}(\U)$ provides a better approximation.
 \end{remark}

\bibliographystyle{alpha}

\end{document}